\documentclass[twoside,letterpaper,draft,10pt]{amsart}
 
 \usepackage{amsmath,amsthm,latexsym,xspace,amscd,amssymb, textcomp}
 \usepackage[margin=1.8in]{geometry}

  \newtheorem{teo}{Theorem}[section]
 \newtheorem{lem}[teo]{Lemma}
 \newtheorem{cor}[teo]{Corollary}
 \newtheorem{prop}[teo]{Proposition}

 \theoremstyle{remark}\newtheorem{remark}[teo]{Remark}
 \theoremstyle{remark}

\begin{document}

\title[A note in algebras with Laurent polynomial identity]{A note in algebras with Laurent\\ polynomial identity}

\author[C.F.Rodrigues]{Claudenir Freire Rodrigues}

\address{Departmento de Matemática, Universidade Federal do Amazonas\\
	Av. General Rodrigo Octavio Jord\~ao Ramos, 1200 - Coroado I, Manaus - AM - Brasil, 69067-005}

\email{claudiomao1@gmail.com}

\keywords{Laurent polynomial identity; unit group; group identity; nil ideal; B. Hartley Conjecture.}


\begin{abstract}
In this article we extend some results about algebra $A$ with the group of units $U(A)$ having special polynomial identity, Laurent polynomial. And we presents a new version of B. Hartley Conjecture with these identities.
\end{abstract}

\maketitle

\section{\textbf{Introduction}}

 A Laurent polynomial $ P = P(x_1,...,x_l) $ in the noncommutative variables $x_i, \, i=1,\ldots,l$ is an element non-zero in the group algebra  $RF_l$ over a ring  $R$  with free group $F_l = < x_1,...,x_l >$.   $ P $ in $RF_l $  is said a Laurent Polynomial Identity  ($LPI$) for a $R$ algebra $A$ (respc. for $U(A)$ the group of units in A) if $P(c_1, ..., c_l)=0 $ for all sequence $c_1,..., c_l$ in $A$ (respc. in U(A)).\\
 A word  $ w $ in a free group is a group identity $(GI)$ for a group $G$ if $w(c_1, ..., c_l)=1$ ($w=1$ for short) for all $c_1,..., c_l$ in $G$ and   $\sum exp _{x}\, w$ denotes the sum of exponents of $w$ in the variable $x$ .  And we say that a nil ideal $I$ has a bounded exponent if there exists a integer $m$ such that for all $a$ in $I$, $a^i=0$ for some $i \leq m$.  \\ 
 
 The following conjecture was introduced by B. Hartley \cite{SKS}, 
 
 \hspace*{0,05cm}\textbf{ Conjecture}:	Let $G$ be a torsion group and $K$ a field. If the group of units
 $U(KG)$ of $KG$ satisfies a group identity $w=1$, then KG satisfies a polynomial
 identity.\\
 
 In \cite{GSV} A. Giambruno, E.Sehgal e A.Valenti settled the conjecture when $K$ is an infinite field, while in \cite{P1}, D.S. Passman gave a necessary and sufficient condition for it to be true and in \cite{LIU} C-H. Liu removed the condition that $K$ is infinite. It is natural to think about extending the conjecture of B. Hartley by considering Laurent identities instead of group identities because $w$ = 1 is a particular case of $LPI$ $P = 1-w$. In this sense in \cite{GonD} M. A. Dokuchaev and J. Z. Gonçalves investigate when the group of units $U(A)$ of an algebraic algebra $A$ over an infinite field $K$, satisfies a Laurent polynomial, showing that if this holds, then $A$
 has a polynomial identity. And in \cite{OJA} O. Broche, J. Z. Gonçalves, and Á. Del Río gave a characterization of group algebras with $LPI$ for the group of units,  satisfying a polynomial identity. Prior to these results in \cite{GJV94} A. Giambruno, E.Jespers and A.Valenti show that the B. Hartley´s Conjecture is true for a group algebra $RG$ over an infinite commutative domain $R$ when $char\, R = 0$  and when $G$ does not contain elements of order $p$ where $char\, R = p$ ($G$ is a p'- group for short). For this and other purposes, the latter authors proved the following crucial result.
 
 \begin{prop}\label{a} 
 	
 	Let $A$ be an algebra over an infinite commutative domain $R$ and suppose that $U(A)$ satisfies a group identity $w=1$. There exists a positive m such that if $ \, a,b,c $ in $A$ and $a^{2}=bc=0$, then $bacA$ is a  nil right ideal of bounded exponent less or equal than m.
 	
 \end{prop}	
 
 This note is structured in the following way. In Section 2  a generalization (Theorem \ref{p1}) of Proposition \ref {a} is made considering the case  $P=\, a_1 + a_2w_2 + \cdots +a_n w_n $ is a $LPI$  for $U(A)$ (the $w_i's$, $i\geq\,2$ are the words non-constant  in $P$) with sum of exponents non-zero in at least one of the variables in every word non-constant and without restrictions about the ring $R$. Notice that the restriction about the exponents is necessary because  according to Remark~\ref{LPInotvalid}  without this assumption the result  not hold.  As a consequence we prove the result of Proposition \ref{a} (Corollary \ref{c1}) with a $LPI$ for $U(A)$ and  we extend   \cite[Lemma 3.2]{LIU} (Corollary \ref{c2}).
 
 In Section 3 we extend    \cite[Corollary 2.2]{GSV}, Corollary \ref{c3}, which will be used to obtain a $LPI$  version for the conjecture of B. Hartley in the terms of the first part of \cite[Theorem 6]{GJV94} from which we remove the condition of infinity in the second part (Theorem \ref{t2}). In this article, we always adopt $LPI$ with the foregoing restriction about the exponents. Also, replacing $x_i$ with $x^{-i}yx^i$ we may assume  $LPI$ is in two variables.

\section{\textbf{Main Results}}
In this section we will demonstrate our main results.

\begin{teo}\label{p1}
	Let $A$ be an algebra whose  group of units $U(A)$ admits a  $LPI$ $P$  over a ring R whith unit  and non-constants words $w_i$ in $P$ has the sum of exponents  non-zero in at least one of the variables. Then there exists a polynomial $f \, \in \, R[X]$ with  degree $d$ determined by $ l=min \, \{\sum exp \,w_i\}$ and $r=max\{\sum exp \, w_i \}$  such that for all $ \, a,b,c, u $ in $A$ with $a^{2}=bc=0$,  $f(bacu)=0$. Thus, $bacA$  has a polynomial identity.
	
\end{teo}

\begin{proof}
	
	Let  $P\,=\, a_1 + a_2w_2 + \cdots + a_t w_t$ the $LPI$ with $w_i $ in the form $$w_i=x_1^{r_1}x_2^{s_2} \cdots x_1^{r_k}x_2^{s_k}$$ with $k\geq 1, r_i,\, s_i $ integers.  We assume that the words are arranged in $P$ so that
	$ \sum exp \, w_i \leq \sum exp \, w_{i+1} $. 
	
	If $ \sum exp \, w_i = \sum exp _{x_1}\, w_i + \sum exp _{x_2}\, w_i = 0 $ then we  substitute $ x_1= x_1^k$ or $x_2= x_2^k$ with $k > 1$ big enough.
	We get a new $LPI$ in the form  $P\,=\, a_1 + a_2{w'}_2 +\cdots+ a_r {w'}_r,$ where   
	$$ \sum exp \, {w'}_i = \sum exp _{x_1}\, {w'}_i + \sum exp _{x_2}\, {w'}_i = k \sum exp _{x_1}\, w_i + \sum exp _{x_2}\, w_i $$  or $$\sum exp \, {w'}_i = \sum exp _{x_1}\, w_i +  k\sum exp _{x_2}\, w_i. $$     In any  case that is not zero. So we suppose $ \sum exp \, w_i $ not zero for all $i = 2,...,t$.		 
	
	From this $$P(\alpha, \alpha)\,=\, a_1 + a_2 \alpha ^l +\cdots+a_i \alpha ^{\sum exp \, w_i} +\cdots+ a_t \alpha ^r =  0   $$ with all powers of $\alpha$ not zero for all $ \alpha \in U(A)$ and  $l$ and $r$ are respectively the minimum and maximum of the sum of the exponents of the  $w_i's$.
	
	After organizing the common powers we will have a polynomial expression of the form	
	$$  f_0(\alpha) = a_1  +  b_l \alpha ^l +\cdots+ b_r \alpha ^r=  0 $$ where $ b_j's$ are parcials sums of $a_i's, \, i=2,\ldots,t $ .	
	
	In  general,  $P(1,1)\,=\, a_1 + a_2 +\cdots + a_t =0$, so $ a_2 + \cdots + a_t \neq 0 $ , because $a_1 \neq 0$. So we can not have all $b_i's =0$ because in this case $a_2 + \cdots +a_r = b_l+ \cdots +b_r=0$. Follows that $f_0$		
	is a non zero polynomial over $R$ with integers exponents not all   necessarily positives.
	
	If $l < 0$  we have that  $\alpha ^{-l}f_0(\alpha)= f_1(\alpha)=0$, where $f_1\neq 0$ is a polynomial identity for $U(A)$  over $R$ of degree at most $r-l$. We can then assume $l>0$ and $f_0$ is a polynomial identity (of degree at most $r$)  for $U(A)$. Now, we are going to use this polynomial with special units in $U(A)$.
	
	First let $b=c$ be, so $a^2=b^2 =0$.    
	Then we have $(1 + aua), \, (1 + bauab) \in U(A)$, for all $u \in A$..		
	
	Thus $\alpha = (1 + aua).(1 + bauab) \, \in \, U(A) $ and for every power  in $f_0$,    
	$$\alpha^h= 1 + aua + bauab +  \sum \textrm{products in}\, aua \, \textrm{and}\, bauab +   (auabauab)^h$$ for all $h= l, \ldots ,r$. \\
	
	As $ a^2=b^2=0$ and $ab(auabauab)^hau=(abau)^{2h +1}$,
	$$ abb_i\alpha^hau = b_iabau +  g_h(abau) + b_i(abau)^{2h+1}  $$  where $g_h$ is a polynomial over $R$ with degree $< 2h+1$. 
	
	It follows that 
	
	\begin{equation}
	\begin{aligned}	
		abf_0(\alpha)au &=  aba_1au + abb_l\alpha ^lau + \dots +ab b_r \alpha ^rau\\ 		
		&= (a_1 + b_l+ \dots + b_r) abau + \sum g_h(abau)+ \sum   b_h(abau)^{2h+1}\\
		&= (a_1 + a_2+ \dots + a_t) abau + \sum g_h(abau)+ \sum   b_h(abau)^{2h+1}\\
		&=  \sum g_h(abau)+ \sum  b_h(abau)^{2h+1}\\		
		&= f_2(abau)=0, \\	
	\end{aligned}
	\end{equation} with  $h= l, \ldots ,r$ and
	  $f_2$ is a non zero polynomial over $R$ of degree $\leq  2r + 1$.  This finish the $b=c$ case.

	In general, if $a,\, \alpha ,\, \beta \in A$ and $a^{2}=\alpha \beta =0$, then  we have $(\beta u \alpha)^{2}=0 \,,$ 	for all $u \in A $. As we saw $f_2$ is a polynomial identity for $(\beta u \alpha)a (\beta u \alpha) A$. Then, $\alpha af_1((\beta u \alpha)a (\beta u \alpha)a)\beta u=f(\alpha a \beta u)=0$ where $f $ (with zero constant term) is a non zero polynomial over $R$ with degree $\leq 2(2r +1 ) +1= 4r +3$ and thus an identity for $\alpha a \beta A$. 
\end{proof}

As a consequence  we prove the following.

\begin{cor}\label{c1} 
	Let $A$ be an algebra over an infinite commutative domain R  whose group of units $U(A)$ has a $LPI$. Then, for all $a,\, b,\, c \in A$ with $a^2=bc=0 $,  $bacA$ is a nil right ideal with bounded exponent.	
	
\end{cor}	

\begin{proof}
	By Theorem \ref{p1}, there exists $f \in R[X]\setminus\{0\}$ with  degree  $ d $, such that $f(bac \lambda u)=0$ for all $\lambda \in R,\, u\, \in \, A$. That is, $\lambda p_1 + \lambda ^2 p_2 + \cdots + \lambda ^d  p_d=0$ where $p_i$ is polynomial on $ bacu $ with $p_d=b_d (bacu)^d$. Now, by the Vandermond argument \cite[Proposition 2.3.27 p.130 ]{R}, $p_1= \dots =p_d=0$, so $ (bacu)^d=0$. 	
\end{proof}
In the last proof  we see that the application of the Vandermonde argument applies with $|R|> d$. We fix notation 
with $f(x)$ and $d$ as in Theorem \ref{p1}.

As another small application, we extend Lemma 3.2  in \cite{LIU} in the context of $LPI$.

\begin{cor}\label{c2}
	
	Let $A$ be an algebra over a field $K$ with $ U(A)$ satisfying a $LPI$, and   $a,\, b \, \in A $ such that $a^2=b^2=0$.
	
	\begin{enumerate}
		
		\item[1.] If $\mid K \mid > 2d$, then $(ab)^{2d}=0.$ 
		\item[2.] If $ab$ is nilpotent, then $(ab)^{2d}=0.$
	\end{enumerate}
	
\end{cor}

\begin{proof}
	Arguing as in the proof of Corollary \ref{c1}, with $ b = c$, we deduce that $(abau)^d=0$ for all $u \in A$ and  with $u=b$, $(ab)^{2d}=0$. Now suppose $ab$ nilpotent, by Theorem \ref{p1} with $( b=c)$ there exists a polynomial $f$ over K  of degree $d$ such that $f(abau)=0$ for all $u \in A$. Thus  $f(aba.b)=f_1(ab)=0$ where $f_1$ is a polynomial of degree $2d$. The minimal polynomial of $ab$ is $g=x^i$, $i>0$, because $ab$ is nilpotent. As $g$ divides $f_1$ implies $i \leq 2d$ and since $g(ab)=(ab)^i=0$, $(ab)^{2d}=0$. 
\end{proof}	

\section{\textbf{ A PLI version for the B. Hartley conjecture}}

First we shall need the following two results.

\begin{lem}\label{l1}
	Let $A$ be an algebra over a commutative domain $R$. Suppose that for all $a,\,b,\,c \in A$ such that $a^2=bc=0$ we have $bac=0$. Then every idempotent of $R^{-1}A$ is central.
\end{lem}		

\begin{proof}
	
	See \cite[Lemma 2.]{GJV94}
	
\end{proof}	

\begin{lem} \cite{GSV} \label{l2}
	Let $A$ be a semiprime ring and let $S= \{ a \in A: for\, all\, b,\,c \in A, bc=0\, implies\,  bac=0 \}$. If $S$ contains all square-zero elements of A, then $S$ contains all nilpotent elements of A.
	
\end{lem}

\begin{proof}
	See \cite[Lemma 2.1.]{GSV}
	
\end{proof}

Let  $R$ be a commutative domain and $A$ a $R$-algebra. $R^{-1}A$ is the ring of fractions of $A$ with respect to the non-zero elements of $R$. 

In the following  a  $LPI$ version of Corollary 2.2 of  \cite{GSV}.
\begin{cor}\label{c3}
	
	Let $A$ be a semiprime algebra over an infinite  commutative domain $R$. If $U(A)$ satisfies a $LPI$, then every idempotent element of $R^{-1} A$ is central and for all $b,\,c \in A$ such that $bc=0$ we have $bac=0$ for every $a$ nilpotent in A.			
	
\end{cor}

\begin{proof} 
	By Corollary \ref{c1} for all $a,\,b,\, c \in A$ with $a^2=bc=0$, $bacA$ is a nil right ideal of bounded exponent. If $bac\neq 0$, by Levitzki´s theorem (see \cite{R}  1.6.26, p.46) A has a non zero nilpotent ideal, against the semiprimeness of A. Thus $bac=0$. Now the result follows  from Lemmas \ref{l1} and  \ref{l2}.
\end{proof}		

In the sequel  the others classical  results  we need.

The $FC$-subgroup of a group $G$ is  given by\\
$\phi (G)=\{ g \in G: g \textmd{ has a finite number os conjugates in}\, G \}$. Let $A$ be a $R$-algebra. We say that $g$ is a multilinear generalized polynomial of degree $n$ if 

$$g(x_1,\ldots,x_n)= \sum_{\sigma \in Sym_n} g^\sigma (x_1,\ldots,x_n)$$
with 
$$g^\sigma (x_1,\ldots,x_n)=\sum ^{a_\sigma}_{j=1} \alpha_{0,\sigma,j}x_{\sigma(1)}\alpha_{1,\sigma,j}x_{\sigma_{(2)}}
\cdots  \alpha_{n-1,\sigma,j}x_{\sigma(n)}\alpha_{n,\sigma,j}$$ 

where $ \alpha_{i,\sigma,j} \in R$ and $a_\sigma$ is some positive number. The foregoing $g$ is said to be nondegenerate if for some $\sigma \in Sym_n$, $g^\sigma$ is not a generalized polynomial identity for $A$. $A$ is said to satisfy a GPI if $A$ satisfies a nondegenerate multilinear generalized polynomial identity.

At next lemmas K is a field and G is a group.

\begin{lem}\label{s}
	
	$K[G]$ satisfies a GPI if and only if $[G : \phi (G)] < \infty$  and $|\phi'(G)|<\infty$. 
	
\end{lem}

\begin{proof}
	See \cite[p. 202]{P}
\end{proof}	

For any integer $n$, the standar polynomial of degree $n$ is given by

$$ S_n(x_1,\ldots,x_n)= \sum_{\sigma \in Sym_n} (sgn \sigma)  x_{\sigma(1)},\ldots,x_{\sigma(n)} $$.

\begin{lem}[Amitsur-Levitzki]\label{l}	
	
	If K is a field. Then $M_m(K)$ satisfies the standar  polynomial  identity $S_{2m}$.
	
\end{lem}

\begin{proof}
	See \cite[Theorem  1.9, p. 175]{P}
\end{proof}	

A result well-known is the lemma:

\begin{lem} \label{f}
	
	Let A be a R-algebra and let I a right ideal of R. If I satisfies a polynomial identity of degree k and $I^k\neq 0$, then A satisfies a GPI.

\end{lem}

\begin{lem}\label{p2}
	
	Suppose that char K = p $>$ 0. Then $K[G]$ is semiprime if and only if $\phi (G)$ is a p'-group.
	
\end{lem}

\begin{proof}
	See \cite[p. 131]{P}
\end{proof}

\begin{lem}\label{p3}	
	
	If char K = 0, then $K[G]$ is semiprime. 
	
\end{lem}

\begin{proof}
	See \cite[p.130]{P}
\end{proof}

\begin{lem}\label{g1}	
	
	Let A be an algebraic algebra over an infinite field K. If $U(A)$ is LPI then A is a PI-ring. Moreover, if A is non-commutative then A is PI-ring, provided that the non-central units of A satisfy an LPI. 
	
\end{lem}

\begin{proof}
	See \cite[Theorem 5]{GonD}
\end{proof}				

Now we are ready to do our adaptations to the LPI case of  the Conjecture of B.Hartley.

\begin{teo}\label{t2}
	Let G be a torsion group, R a commutative domain and  U(RG) satisfies a Laurent polynomial identity. 
	
	\begin{enumerate}
		\item[1.] If either char R = 0 or char R = p $>$0 with G a p'-group and R infinite, then RG satisfies $S_4$, the standard identity of degree four;
		
		\item[2.] If G is any group, then  either there exists a integer d $> 0$ such that for all a,b,c $\in$ RG with $a^2=bc=0$, we have that bacRG is a nilpotent ideal of exponent less than d  or $\phi (G)$ is of finite index in G and $\phi (G)'$ is a finite group. In the latter case, with G a torsion group and R an infinite field, RG has a polynomial indentity.
		
	\end{enumerate}

\end{teo}

\begin{proof}  By Lemmas \ref{p2} and \ref{p3} $KG$ is a semiprime $K$-algebra and by Lema \ref{c3} for all $g \in G$
	with order $n$, the idempotent $n^{-1}$ $(1 + g + g^2 +\ldots +g^{n-1})$ is  central in $KG$, $K$ is the field of fractions of R. From this $\langle g \rangle$ is a normal cyclic group  in $G$. Thus $G$ is an abelian or Hamiltonian group. With $G$ Hamiltonian, $G = Q_8\times E \times A$, where $Q_8$ is the quarternion group of order 8, $E$ is an elementary Abelian 2-group and $A$ is abelian with elements of odd order. Thus $G$ has an abelian subgroup $S$ of index 2. And by \cite[Lemma V.1.10]{P} $KG$ embeds in $M_2(KS)$ and by Lemma \ref{l}, $KG$  and so, $RG$ satisfies $S_4$.

	For the second part if  $a,b,c \in RG$ with $a^2=bc=0$. By Theorem \ref{p1}  $bacRG$ has a polynomial identity of degree $d$. If $(bacRG)^d \neq 0 $, then by Lemma \ref{f} $RG$ satisfies a GPI and from this $KG$ satisfies a GPI.  By  Lemma \ref{s} $[G :\phi (G)] < \infty $ and $|\phi '(G)| < \infty$. In this case with $G$ a torsion group $\phi (G)$ is locally finite. Follows that $G$ is locally finite and, so, $RG$ is algebraic. With R an infinite field by Lemma \ref{g1} $RG$ has a polynomial identity.
\end{proof}

\begin{remark}\label{LPInotvalid}
	
	In  general,  Theorem \ref{p1}  not hold without the restrictions  about exponents. For example, let $A=M_n(R), \, n>1$ and $R$ an infinite commutative domain. By Lemma \ref{l} $U(A)$ satisfies the $LPI$ 
	
	$$ f(x_1,\ldots,x_{2n})= S_{2n}\cdot(x_1\cdots x_{2n})^{-1}  $$  where $S_{2n}$ is the  standard identity and $f$ is a  $LPI$ with all words non-constant  having sum of exponents zero at every variable. Admitting the Theorem \ref{p1}, $ bacu$ is algebraic for all $u \in A$, in particular, for $ c=b $ and $u=a$, $ba$ is algebraic.  It follows that by Vandermonde augument  $ba$ is nilpotent. In general,   $ba$ is not a nilpotent element.   
	We can see this with  $a=e_{21}$ and $b=e_{12}$,  $ba=e_{11}$ is not nilpotent.
	
	Still considering $ A = M_n (R) $, $R$ a infinite commutative domain, $ a = e_ {21} $ and $ b =e_ {12} $ by Proposition \ref {a} $ U (A) $ does not admit group identity. This show that there are algebras with  $LPI $ without $ GI $.
	
\end{remark}

\newpage

\printindex

\end{document}